\documentclass[a4paper, reqno, 11pt]{amsart}

\usepackage[english]{babel}
\usepackage{amsmath}
\usepackage{mathrsfs}
\usepackage{amssymb}
\usepackage{amsthm}
\usepackage{enumerate}
\usepackage{ifthen}
\usepackage{bbm}
\usepackage{color}
\usepackage{xcolor}
\usepackage{graphicx}
\provideboolean{shownotes} 
\setboolean{shownotes}{true}
\usepackage{hyperref}

\def\Xint#1{\mathchoice
{\XXint\displaystyle\textstyle{#1}}%
{\XXint\textstyle\scriptstyle{#1}}%
{\XXint\scriptstyle\scriptscriptstyle{#1}}%
{\XXint\scriptscriptstyle\scriptscriptstyle{#1}}%
\!\int}
\def\XXint#1#2#3{{\setbox0=\hbox{$#1{#2#3}{\int}$ }
\vcenter{\hbox{$#2#3$ }}\kern-.6\wd0}}

\def\dashint{\Xint-}

\usepackage{geometry}
\newcommand{\margnote}[1]{
\ifthenelse{\boolean{shownotes}}%
{\marginpar{\raggedright\tiny\texttt{#1}}}%
{}%
}

\newcommand{\hole}[1]{
\ifthenelse{\boolean{shownotes}}%
{\begin{center} \fbox{ \rule {.25cm}{0cm}
\rule[-.1cm]{0cm}{.4cm} \parbox{.85\textwidth}{\begin{center}
\texttt{#1}\end{center}} \rule {.25cm}{0cm}}\end{center}}
{}
}
\newtheorem{thm}{Theorem}[section]

\newtheorem{prop}[thm]{Proposition}
\newtheorem{lem}[thm]{Lemma}

\newtheorem{rem}[thm]{Remark}

\theoremstyle{definition}
\newtheorem{defn}[thm]{Definition}

\newcommand{\e}{\varepsilon}		       
\newcommand{\R}{\mathbb{R}}
\newcommand{\T}{\mathbb{T}^2}
\newcommand{\N}{\mathbb{N}}
\newcommand{\Z}{\mathbb{Z}}
\newcommand{\dive}{\mathop{\mathrm {div}}}
\newcommand{\curl}{\mathop{\mathrm {curl}}}
\newcommand{\weakto}{\rightharpoonup}
\newcommand{\weaktos}{\stackrel{*}{\rightharpoonup}}
\newcommand{\de}{\,\mathrm{d}}
\newcommand{\Lip}{\mathrm{Lip}}
\DeclareMathOperator*{\esssup}{ess\,sup}

\numberwithin{equation}{section}


\begin{document}

\title[Energy conservation for 2D Euler with vorticity in $L(\log L)^\alpha$]{Energy conservation for 2D Euler with vorticity in $L(\log L)^\alpha$}

\author[G. Ciampa]{Gennaro Ciampa}
\address[G. Ciampa]{Dipartimento di Matematica ``Tullio Levi Civita''\\ Universit\`a degli Studi di Padova\\Via Trieste 63 \\35131 Padova \\ Italy}
\email[]{\href{ciampa@}{ciampa@math.unipd.it}, \href{gennaro.ciampa@}{gennaro.ciampa@unipd.it}}

\begin{abstract}
In these notes we discuss the conservation of the energy for weak solutions of the two-dimensional incompressible Euler equations. Weak solutions with vorticity in $L^\infty_t L^p_x$ with $p\geq 3/2$ are always conservative, while for less integrable vorticity the conservation of the energy may depend on the approximation method used to construct the solution. Here we prove that the canonical approximations introduced by DiPerna and Majda provide conservative solutions when the initial vorticity is in the class $L(\log L)^\alpha$ with $\alpha>1/2$.
\end{abstract}

\maketitle

\section{Introduction}
The motion of an incompressible, homogeneous, planar fluid is described by the system of the 2D Euler equations
\begin{equation}\label{eq:eu}
\begin{cases}
\partial_t u+\left(u\cdot\nabla\right)u+\nabla p=0,\\
\dive u=0,\\
u(0,\cdot)=u_0,
\end{cases}
\end{equation}
where $u:[0,T]\times\R^2\to\R^2$ is the velocity of the fluid, $p:[0,T]\times\R^2\to\R$ is the pressure and $u_0:\R^2\to\R$ is a given initial configuration. The first set of equations derive from Newton's second law while the divergence-free condition expresses the conservation of mass. A peculiar fact of the 2D case is that the vorticity $\omega$, defined as
$$
\omega=\partial_{x_1}u_2-\partial_{x_2}u_1,
$$
is a scalar quantity which is advected by the velocity $u$. In fact, the equations \eqref{eq:eu} can be rewritten in the vorticity formulation
\begin{equation}\label{eq:vort}
\begin{cases}
\partial_t\omega+u\cdot\nabla\omega=0,\\
u=K*\omega,\\
\omega(0,\cdot)=\omega_0,
\end{cases}
\end{equation}
where $K(x)=x^\perp/(2\pi |x|^2)$ is the 2D Biot-Savart kernel. Note that the equation \eqref{eq:vort} is a non-linear and non-local transport equation.\\

The well-posedness of \eqref{eq:eu} is an old and outstanding problem. For smooth initial data, the existence and uniqueness of classical solutions was proved in \cite{L, W}. The existence of weak solutions has been proved by DiPerna and Majda in \cite{DPM} by assuming that the initial datum $\omega_0\in L^1\cap L^p(\R^2)$ with $1<p\leq\infty$. Besides this result, the goal of \cite{DPM} was to develop a rigorous framework for the study of {\em approximate solution sequences} of the two-dimensional Euler equations. In particular, the authors proved a general compactness theorem towards measure-valued solutions by assuming that $\omega_0$ is a {\em vortex-sheet}, i.e. $\omega_0\in\mathcal{M}\cap H^{-1}_{\mathrm{loc}}(\R^2)$. They described three different methods to construct approximate solution sequences:
\begin{itemize}
\item[(ES)] approximation by exact smooth solutions of \eqref{eq:eu};
\item[(VV)] vanishing viscosity from the two-dimensional Navier-Stokes equations;
\item[(VB)] vortex-blob approximation. 
\end{itemize} 
In \cite{DPM} DiPerna and Majda showed the existence of weak solutions via a compactness argument based on the methods (ES) and (VV). The counterpart for the vortex-blob method was proved by Beale in \cite{Be}. In these results, the $L^p$-integrability with $1< p\leq \infty$ of $\omega_0$ is crucial in order to use Sobolev embeddings which guarantee the strong compactness in $L^2$ of an approximate solution sequence. This is enough to deal with the non-linear term in the equations. However, in the case $\omega_0$ is just $L^1$ or a measure with distinguished sign, it turns out that the limit vector field is a solution of \eqref{eq:eu} even though concentrations may occur in the non-linearity. This is a purely 2D phenomenon known as \emph{concentration-cancellation}, and it was studied in \cite{De, VW}. \\
The uniqueness of weak solutions in the class considered in \cite{DPM} is still an open problem, contrary to the case $p=\infty$ which has been proved by Yudovich \cite{Y}. There exist several partial results towards the non-uniqueness in the case of unbounded initial vorticity, see \cite{BM, BS, MeS, VI, VII}.
\\

Smooth solutions of \eqref{eq:eu} are known to be {\em conservative}, which means that $\|u(t)\|_{L^2}=\|u_0\|_{L^2}$ for all times, while this property is not trivial when we consider weak solution. The problem of the energy conservation, assuming only integrability conditions on the vorticity, has been addressed in \cite{CFLS}: the authors consider the 2D Euler equations on the two-dimensional flat torus $\T$ and they prove that \emph{all} weak solutions satisfy the energy conservation if the vorticity $\omega\in L^\infty((0,T);L^p(\T))$ with $p\geq 3/2$. The proof is based on a mollification argument and the exponent $p=3/2$ is required in order to have weak continuity of a commutator term in the energy balance. The authors also give an example of the sharpness of the exponent $p=3/2$ in their argument, but still leaves open the question of the existence of non-conservative solutions below this integrability threshold. Moreover, they show that if $\omega\in L^\infty((0,T);L^p(\T))$, with $1<p<3/2$, solutions constructed via (ES) and (VV) conserve the kinetic energy.
\\

Here we discuss the conservation of the energy for solutions of the 2D Euler equations when the initial vorticity is slightly more than integrable, namely $\omega_0\in L^1\cap L(\log L)^\alpha(\R^2)$ with $\alpha >1/2$. Existence of weak solutions of \eqref{eq:eu} in this setting was proved by Chae, first in \cite{Ch} in the case $\omega_0\in L^1\cap L\log L(\R^2)$, and then extended to the case $\omega_0\in L^1\cap L(\log L)^{1/2}(\R^2)$ in \cite{Ch2}. In these results, the strategy of the proof is based on the properties of Calder\'on-Zygmund singular integral operators and compact embeddings of Orlicz-Sobolev spaces into $L^2_\mathrm{loc}(\R^2)$.

In a similar fashion to the framework of DiPerna and Majda, in \cite{FLT} the authors introduce the definition of $H^{-1}_\mathrm{loc}$-stability for a sequence of approximating vorticity $\omega^\e$, showing that it is a sharp criterion for the strong $L^2_{\mathrm{loc}}$-convergence of an approximate solution sequence $u^\e$. With their approach they are able to recover previous existence results, expanding the set of possible initial data to much more general {\em rearrangement invariant} spaces, such as the Orlicz spaces $L(\log L)^\alpha$, with $\alpha\geq 1/2$, and the Lorentz spaces $L^{(1,q)}$ with $1<q\leq 2$.
 
Finally, in \cite{LMP} it has been proven that the strong $L^2$-compactness of a sequence of velocity fields constructed via (VV) is equivalent to the energy conservation property. In virtue of this result, by posing the problem in the two-dimensional torus, the authors obtained as a corollary that the vanishing viscosity limit produce conservative weak solutions for initial vorticity in the rearrangement invariant spaces considered in \cite{FLT}, including $L(\log L)^\alpha$ with $\alpha>1/2$. \\

The contribution of these notes in the theory of conservative weak solutions of \eqref{eq:eu} is the following: we consider an initial datum $u_0\in L^2(\R^2)$ such that $\omega_0\in L(\log L)^\alpha(\R^2)$ with compact support and we prove that the canonical approximations introduced in \cite{DPM} produce approximate solution sequences such that the velocity converges {\em globally} in $L^2$ if $\alpha>1/2$. This allows us to prove that the vortex-blob method yields to conservative weak solutions and, in this setting, we extend the results of \cite{CFLS, LMP} concerning (ES) and (VV) to the case in which the domain is the whole plane $\R^2$. In order to get the strong convergence in $C([0,T];L^2(\R^2))$ of the approximating velocity, we will exploit the techniques of \cite{CCS3, Sc} by adapting the Serfati identity to this less integrable setting. In particular, it would be crucial that the approximating vorticity converge strongly in $C([0,T];L^1(\R^2))$, as shown recently in \cite{CCS4, CCS3}.

\medskip
\section{The two-dimensional Euler equations}
The goal of this section is to provide some prelimiary results on weak solutions of the 2D Euler equations. First, we introduce the notations used in the paper. Then, we will pay particular attention to the theory developed by DiPerna and Majda in \cite{DPM}. Finally, we will summarize some more recent results concerning conservative weak solutions.
\subsection{Notations}
We will denote by $L^p(\R^d)$ the standard Lebesgue spaces and with $\|\cdot\|_{L^p}$ their norm. Moreover, $L^p_c(\R^d)$ denotes the space of $L^p$ functions defined on $\R^d$ with compact support. The Sobolev space of $L^p$ functions with distributional derivatives of first order in $L^p$ is denoted by $W^{1,p}(\R^d)$. The spaces $L^p_{\mathrm{loc}}(\R^d),W^{1,p}_{\mathrm{loc}}(\R^d)$ denote the space of functions which are locally in $L^p(\R^d),W^{1,p}(\R^d)$ respectively. We will denote by $H^1(\R^d)$ the space $W^{1,2}(\R^d)$ and by $H^{-1}(\R^d)$ its dual space. Moreover, we will say that a function $u$ is in $H^{-1}_{\mathrm{loc}}(\R^d)$ if $\rho u\in H^{-1}(\R^d)$ for every function $\rho\in C^\infty_c(\R^d)$. We denote with $L(\log L)^\alpha(\R^d)$ the space of functions $f$ such that
$$
\int_{\R^d}|f(x)|(\log^+(|f(x)|))^\alpha\de x< \infty,
$$
endowed with the Luxemburg norm
\begin{equation}
\|f\|_{L(\log L)^\alpha}=\inf\left\{k>0:\int_{\R^d}\frac{|f|}{k}\left(\log^+\left( \frac{|f|}{k} \right)\right)^\alpha\de x\leq 1  \right\},
\end{equation}
where the function $\log^+$ is defined as
$$
\log^+(t)=\begin{cases}
\log(t)&\mbox{if }t\geq 1,\\
0&\mbox{otherwise},
\end{cases}
$$
and $L(\log L)^\alpha_c(\R^d)$ will be the space of functions in $L(\log L)^\alpha(\R^d)$ with compact support. We denote by $L^p((0,T);L^q(\R^d))$ the space of all measurable functions $u$ defined on $[0,T]\times\R^d$ such that
$$
\|u\|_{L^p((0,T);L^q(\R^d))}:=\left(  \int_0^T\|u(t,\cdot)\|^p_{L^q} \de t\right)^{\frac{1}{p}}<\infty,
$$
for all $1\leq p<\infty$, and
$$
\|u\|_{L^\infty((0,T);L^q(\R^d))}:=\esssup_{t\in [0,T]}\|u(t,\cdot)\|_{L^q}<\infty,
$$
and analogously for the spaces $L^p((0,T);W^{1,q}(\R^d))$. We denote by $B_R$ the ball of radius $R>0$ centered in the origin of $\R^d$. In the estimates we will denote with $C$ a positive constant which may change from line to line.
Finally, it is useful to denote with $\star$ the following variant of the convolution
\begin{align*}
v\star w&=\sum_{i=1}^2 v_i*w_i \hspace{1cm}\mbox{if }v,w\mbox{ are vector fields in }\R^2,\\
A\star B&=\sum_{i,j=1}^2A_{ij}*B_{ij} \hspace{0.5cm}\mbox{if }A,B\mbox{ are matrix-valued functions in }\R^2.
\end{align*}
With the notations above it is easy to check that if $f:\R^2\to\R$ is a scalar function and $v:\R^2\to\R^2$ is a vector field, then
$$
f*\curl v=\nabla^\perp f\star v,
$$
$$
\nabla^\perp f\star \dive(v\otimes v)=\nabla\nabla^\perp f\star(v\otimes v).
$$
\vspace{0.3cm}
\subsection{Weak solutions}
We recall the definition of weak solution of the Euler equations as in \cite{DPM}.
\begin{defn}\label{def:weaksoleu}
A vector valued function $u \in L^\infty((0,T);L^2_{\mathrm{loc}}(\R^2))$ is a {\em weak solution} of $(\ref{eq:eu})$ if it satisfies:
\begin{enumerate}
\item for all test functions $\Phi\in C^\infty_0((0,T)\times\R^2)$ with $\dive\Phi=0$,
\begin{equation}\label{eq:weakeu}
\int_0^T\int_{\R^2} \left( \partial_t \Phi \cdot u+ \nabla \Phi : u\otimes u \right)\de x \de t=0;
\end{equation}
\item $\dive u=0$ in the sense of distributions;
\item $u\in \mathrm{Lip}([0,T);H^{-L}_{\mathrm{loc}}(\R^2))$ for some $L>0$ and $u(0,x)=u_0(x)$.
\end{enumerate}
\end{defn}
\begin{rem}
The choice of divergence-free test functions allow to not consider the pressure in the weak formulation \eqref{eq:weakeu}. It can be formally recovered by the formula
$$
-\Delta p=\dive\dive(u\otimes u),
$$
which is obtained applying the divergence in the momentum equation in \eqref{eq:eu}.
\end{rem}
In \cite{DPM}, DiPerna and Majda introduced the following definition of an {\em approximate solution sequence} of the 2D Euler equations.
\begin{defn}\label{def:appsolseq}
A sequence of smooth velocity fields $u^n$ 
with vorticity $\curl u^n=\omega^n\in C([0,T];L^1(\R^2))$ is an {\em approximate solution sequence} for the 2D Euler equations provided that
\begin{itemize}
\item[(i)] $u^n$ has uniformly bounded local kinetic energy and $u^n$ is incompressible, i.e., for each $R>0$ and $T>0$, there exists $C(R)>0$ such that
$$
\max_{t\in[0,T]}\int_{B_R}|u^n(t,x)|^2\de x\leq C(R),\hspace{1cm}\dive u^n=0;
$$
\item[(ii)] the vorticity $\omega^n$ is uniformly bounded in $L^1$, i.e., for every $T>0$,
$$
\max_{t\in[0,T]}\int_{\R^2}|\omega^n(t,x)|\de x\leq C;
$$
\item[(iii)] for some $L>0$, the sequence $u^n$ is uniformly bounded in $\Lip([0,T];H^{-L}_{\mathrm{loc}}(\R^2))$;
\item[(iv)] $u^n$ is weakly consistent with the 2D Euler equations, i.e.
\begin{equation}
\lim_{n\to \infty}\int_0^T\int_{\R^2}\left(\partial_t\Phi \cdot u^n+\nabla\Phi:u^n\otimes u^n\right)\de x \de t=0,
\end{equation}
for every $\Phi\in C^\infty_c((0,T)\times\R^2)$ with $\dive\Phi=0$.
\end{itemize}
\end{defn}
Besides the very general definition, in \cite{DPM} the authors give three different examples of approximate solutions sequences, which are important for physical or numerical reasons. They are the following.\\
\begin{itemize}
\item[\textbf{(ES)}] \textbf{Approximation by exact smooth solutions of \eqref{eq:eu}.} We consider a smooth approximation of the initial datum $u_0^\delta$ such that $u_0^\delta\to u_0$ in $L^2_\mathrm{loc}$ and we define $u^\delta$ the unique solution of the approximating problem
\begin{equation}
\begin{cases}
\partial_t u^\delta+(u^\delta\cdot\nabla)u^\delta+\nabla p^\delta=0,\\
\dive u^\delta=0,\\
u^\delta(0,\cdot)=u_0^\delta.
\end{cases}
\end{equation}
Then, a solution $u$ of \eqref{eq:eu} is constructed analyzing the limit of the sequence $u^\delta$ as $\delta\to 0$.
\\
\item[\textbf{(VV)}] \textbf{Vanishing viscosity from the two-dimensional Navier-Stokes equations.} We consider the two-dimensional incompressible Navier-Stokes equations
\begin{equation}\label{eq:ns}
\begin{cases}
\partial_t u^\nu+(u^\nu\cdot\nabla)u^\nu+\nabla p^\nu=\nu\Delta u^\nu,\\
\dive u^\nu=0,\\
u^\nu(0,\cdot)=u_0^\nu,
\end{cases}
\end{equation}
where $\nu>0$ is the viscosity of the fluid and $u_0^\nu$ is smooth and converges in $L^2_\mathrm{loc}$ towards $u_0$ as $\nu\to 0$. Then, a solution $u$ of \eqref{eq:eu} is constructed analyzing the vanishing viscosity limit of the sequence $u^\nu$.
\\
\item[\textbf{(VB)}] \textbf{Vortex-blob approximation.} It is a numerical method which is the prototype of several important numerical schemes. It is based on the idea of approximating the vorticity with a finite number of cores which evolve according to the velocity of the fluid. Without going into details, the approximating velocity $u^\e$ solves the system
\begin{equation}
\begin{cases}
\partial_t u^\e+\left(u^\e\cdot\nabla\right) u^\e+\nabla p^\e=K*E_\e,\\
\dive u^\e=0,\\
u^\e(0,\cdot)=u_0^\e,
\end{cases}
\end{equation}
where $u_0^\e$ is a suitable smooth approximation of the initial datum and $E_\e$ is an error term which comes from the fact that, roughly speaking, each blob is rigidly translated by the flow.
We give the precise construction together with its main properties in the Appendix. 
\end{itemize}
\vspace{0.3cm}
By assuming only integrability hyphotesis on the initial vorticity $\omega_0$, the existence of weak solutions constructed with the methods above has been proven in \cite{Be, DPM}. For simplicity of exposition, for the remainder of this subsection we will use $n$ as an approximation parameter for all the three methods.
\begin{thm}\label{teo:appsolseq}
Let $u_0\in L^2_{\mathrm{loc}}(\R^2)$ be a divergence-free vector field vanishing uniformly as $|x|\to\infty$ and let $\omega_0=\curl u_0\in L^p_c(\R^2)$ for some $p> 1$. Let $u^n$ be an approximate solution sequence constructed via one of the methods (ES), (VV), (VB), where the associated initial datum $u^n_0\to u_0$ in $L^2_{\mathrm{loc}}(\R^2)$. Then, there exists a subsequence of $u^n$ and a vector field $u \in L^\infty((0,T);L^2_{\mathrm{loc}}(\R^2))\cap \Lip([0,T];H^{-L}_{\mathrm{loc}}(\R^2))$ which vanishes uniformly as $|x|\to \infty$ with the following properties:
\begin{itemize}
\item $u(0,\cdot)=u_0$,
\item $u^n\to u$ in $L^2((0,T);L^2_{\mathrm{loc}}(\R^2))$,
\item $\omega^n\weaktos \omega$ in $L^\infty((0,T);L^p(\R^2))$,
\item $\omega^n\to\omega$ in $C([0,T];H^{-L-1}_{\mathrm{loc}}(\R^2))$.
\end{itemize}
\end{thm}
\begin{rem}
Note that the setting of the previous theorem is for a regime where the uniqueness of solutions of \eqref{eq:eu} is not known. Therefore, the three methods could have multiple limit points which may also change depending on the approximation.
\end{rem}
As already mentioned in the introduction, the previous theorem has been generalized by Chae \cite{Ch, Ch2}:
\begin{thm}\label{teo:chae}
Let $u_0\in L^2_{\mathrm{loc}}(\R^2)$ be a divergence-free vector field such that $\curl u_0=\omega_0\in L(\log L)^\alpha_c(\R^2)$ with $\alpha\geq 1/2$. Then, there exists a weak solution $u$ of \eqref{eq:eu} with initial datum $u_0$ satisfying
\begin{equation}
u\in C([0,T];L^2_{\mathrm{loc}}(\R^2)).
\end{equation}
\end{thm}
The proof of Theorem \ref{teo:chae} strongly relies on the fact that the operator
\begin{equation}
T:f\in L(\log L)^\alpha_c(\R^2)\to K*f\in L^2_{\mathrm{loc}}(\R^2),
\end{equation}
is compact for $\alpha>1/2$, where $K$ is the two dimensional Biot-Savart kernel. It is worth to note that the solutions are constructed analyzing the vanishing viscosity limit of the corresponding Navier-Stokes equations with the same initial data. Moreover, we remark that in \cite{Sc} it is shown that it is possible to construct a function $f$ which belongs to $L(\log L)^\alpha_c(\R^2)$ with $\alpha<1/2$ such that $K*f$ is not locally square integrable.
\vspace{0.5cm}

We conclude this subsection by summarizing some known results about the strong convergence in $C(L^p)$ of the approximating vorticity $\omega^n$. This problem has been addressed by several authors  in different settings, especially with regard to the inviscid limit of the Navier-Stokes equations, see for example \cite{CCS4, CDE, NSW}.
We collect the results we need in the following theorem, see \cite{BBC2, CCS4, CCS3}.
\begin{thm}\label{teo:convforte}
Let $\omega_0\in L^p_c(\R^2)$ with $p\geq 1$ and let $\omega^n$ be a sequence of approximating vorticity constructed via one of the three methods (ES), (VV), or (VB). Then, there exists $\omega\in C([0,T];L^1\cap L^p(\R^2))$ such that
\begin{equation}
\omega^n\to\omega \hspace{0.5cm}\mbox{ in } C([0,T];L^1\cap L^p(\R^2)).
\end{equation}
\end{thm}
\begin{rem}\label{rem:conv_vort}
Being $L(\log L)^\alpha_c\subset L^1_c$, assuming $\omega_0\in L(\log L)^\alpha_c$ by Theorem \ref{teo:convforte} if $\omega^n$ is a sequence constructed via one of the aforementioned methods, then there exists $\omega\in C([0,T];L^1(\R^2))$ such that
\begin{equation}
\omega^n\to\omega \hspace{0.5cm}\mbox{ in } C([0,T];L^1(\R^2)).
\end{equation}
\end{rem}
\bigskip

\subsection{Conservative solutions}
In this subsection we discuss the conservation of the energy for the 2D Euler equations. We recall the following definition.
\begin{defn}
Let $u\in C([0,T];L^2(\R^2))$ be a weak solution of \eqref{eq:eu} with initial datum $u_0\in L^2(\R^2)$. We say that $u$ is a \emph{conservative weak solution} if
$$
\|u(t,\cdot)\|_{L^2}=\|u_0\|_{L^2} \hspace{1cm} \forall t\in [0,T].
$$
\end{defn}
It is well-known that in the two-dimensional case, even if we assume that the vorticity is bounded, the velocity field is in general not square integrable. In order to define the kinetic energy, we need to require that the vorticity has zero mean value, see \cite{MB}.
\begin{prop}\label{prop:dpm}
An incompressible velocity field in $\R^2$ with vorticity of compact support has finite kinetic energy if and only if the vorticity has zero mean value, that is
\begin{equation}\label{zeromean}
\int_{\R^2}|u(x)|^2\de x<\infty\iff\int_{\R^2}\omega(x)\de x=0.
\end{equation}
\end{prop}
As already explained in the introduction, the problem of the conservation of the energy in this low regularity setting has been addressed in \cite{CFLS}: they showed that every weak solution is conservative if $\curl u\in L^\infty((0,T);L^p(\T))$ with $p\geq 3/2$, while for less integrable vorticities the conservation of the energy may depend on the approximation procedure. In particular, by collecting the results of \cite{CFLS, CCS4, CCS3} we have the following theorem.\\
\begin{thm}\label{teo:cons_en}
Let $u\in C([0,T];L^2(\R^2))$ be a weak solution of \eqref{eq:eu} with $\omega=\curl u$ satisfying \eqref{zeromean}. Then,
\begin{itemize}
\item if $\omega\in L^\infty((0,T);L^1\cap L^{\frac{3}{2}}(\R^2))$, then $u$ is conservative;
\item if $\omega\in L^\infty((0,T);L^1\cap L^p(\R^2))$ with $p>1$, and $u$ is constructed as limit of one of the approximations (ES), (VV), or (VB), then $u$ is conservative.
\end{itemize}
\end{thm}
We finish this subsection by recalling a theorem that has been proved in \cite{LMP}. It characterizes the compactness of (VV) and the energy conservation in terms of the classical structure function
$$
S^T_2(u;r):=\left( \int_0^T\int_{\T}\dashint_{B_r}|u(t,x+h)-u(t,x)|^2 \de h\de x \de t \right)^{1/2}. 
$$
The main statement from \cite{LMP} is the following.
\begin{thm}\label{teo:mishra}
Let $u^\nu$ be the unique solution of \eqref{eq:ns} with a smooth initial datum $u_0^\nu$ such that 
$$
u_0^\nu\to u_0 \hspace{0.3cm}\mbox{ in }L^2(\T).
$$
Let $u\in L^\infty((0,T);L^2(\T))$ be a solution of \eqref{eq:eu} with initial datum $u_0$ such that, up to a sub-sequence,
$$
u^\nu\weakto u \hspace{0.3cm}\mbox{ in }L^2(\T).
$$
Then the following are equivalent:
\begin{itemize}
\item[(i)] $u^\nu\to u $ strongly in $L^p((0,T);L^2(\T))$ for some $1\leq p<\infty$,
\item[(ii)]  there exists a bounded modulus of continuity $\phi(r)$ such that, uniformly in $\nu$,
$$
S^T_2(u^\nu;r) \leq \phi(r)\hspace{0.3cm}\forall r\geq 0,
$$
\item[(iii)] $u$ is a conservative weak solution.
\end{itemize}
\end{thm}
It is important to note that, by using the result in \cite{FLT}, the previous theorem implies that solutions constructed via (VV) are conservative if $\omega_0\in L(\log L)^\alpha(\T)$. Then, our Theorem \ref{thm:cons_vv} will extend the aforementioned result to the class $\omega_0\in L(\log L)_c^\alpha(\R^2)$.
\begin{rem}\label{rem:mishra}
The Theorem \ref{teo:mishra} holds even if we replace the two-dimensional torus $\T$ with the whole plane $\R^2$, taking into account the appropriate technical considerations.
\end{rem}

\bigskip

\section{A priori estimates}
In this section we summarize some a priori estimates for the approximating vorticity constructed via the approximation methods introduced in Section 2.2. We will always assume that $\omega_0\in L(\log L)^\alpha_c(\R^2)$ with $\alpha>1/2$. As already stressed in the introduction, these estimates will be crucial in order to address the strong convergence of the velocity field in $C([0,T];L^2(\R^2))$, which will be the topic of the next section.
\subsection{Limit of exact smooth solutions}
Let $\rho_\delta$ be a standard smooth mollifier and consider the following Cauchy problem
\begin{equation}\label{eq:ESvortCauchy}
\begin{cases}
\partial_t\omega^\delta+v^\delta\cdot\nabla\omega^\delta=0,\\
v^\delta=K*\omega^\delta,\\
\omega^\delta(0,\cdot)=\omega_0^\delta,
\end{cases}
\end{equation}
where $\omega_0^\delta=\omega_0*\rho_\delta$. We have the following.
\begin{lem}\label{lem:es}
Let $\omega^\delta$ be the unique smooth solution of \eqref{eq:ESvortCauchy}. Then,
\begin{equation}
\sup_{t\in[0,T]}\int_{\R^2}|\omega^\delta(t,x)|(\log(e+|\omega^\delta(t,x)|))^\alpha\de x\leq C,
\end{equation}
where $C$ is a positive constant which does not depend on $\delta$.
\end{lem}
\begin{proof}
Define $\beta(s)=s(\log(e+s))^\alpha$ and multiply the equations in \eqref{eq:ESvortCauchy} by $\beta'(|\omega^\delta|)$. Then, by integrating in space and time we get that
\begin{equation}\label{eq:bilancioES}
\frac{\de}{\de t}\int_{\R^2}\beta(|\omega^\delta(t,x)|)\de x= 0.
\end{equation}
By using the convexity of $\beta$ and Jensen's inequality, it follows that
$$
\int_{\R^2}\beta(|\omega_0^\delta|)\de x\leq\int_{\R^2}\beta(|\omega_0|)\de x\leq C\left(\|\omega_0\|_{L^1}+\int_{\R^2}|\omega_0|(\log^+(|\omega_0|))^\alpha \de x\right)<\infty,
$$
and then, by integrating in time in \eqref{eq:bilancioES} we have the result.
\end{proof}
\vspace{0.5cm}

\subsection{The vanishing viscosity limit}
We now deal with the vanishing viscosity limit of the Navier-Stokes equations. Let $\rho_\nu$ a standard smooth mollifier and let $\omega^\nu$ the solution of 
\begin{equation}\label{eq:VVvortCauchy}
\begin{cases}
\partial_t\omega^\nu+v^\nu\cdot\nabla\omega^\nu=\nu\Delta\omega^\nu,\\
v^\nu=K*\omega^\nu,\\
\omega^\nu(0,\cdot)=\omega_0^\nu,
\end{cases}
\end{equation}
where $\omega_0^\nu=\omega_0*\rho_\nu$. We have the following
\begin{lem}
Let $\omega^\nu$ be the unique smooth solution of \eqref{eq:VVvortCauchy}. Then,
\begin{equation}
\sup_{t\in[0,T]}\int_{\R^2}|\omega^\nu(t,x)|(\log(e+|\omega^\nu(t,x)|))^\alpha\de x\leq C,
\end{equation}
where $C$ is a positive constant which does not depend on $\nu$.
\end{lem}
\begin{proof}
We just sketch the proof since it is very similar to the one of Lemma 3.1. Define $\beta(s)=s(\log(e+s))^\alpha$ and multiply the equations in \eqref{eq:VVvortCauchy} by $\beta'(|\omega^\nu|)$. Then, by integrating in space and time we get that
\begin{equation}\label{eq:bilancioVV}
\frac{\de}{\de t}\int_{\R^2}\beta(|\omega^\nu(t,x)|)\de x= -\nu\int_0^T\int_{\R^2}|\nabla\omega^\nu(t,x)|^2\beta''(|\omega^\nu(t,x)|)\de x \de t\leq 0,
\end{equation}
since $\beta$ is convex. Then, integrating in time \eqref{eq:bilancioVV} we have the result.
\end{proof}
\vspace{0.5cm}
\subsection{The vortex-blob method}
We finally deal with the vortex-blob method. The reader can find the precise definition of the vortex-blob method and some of its properties in the Appendix at the end of this note.
\begin{lem}
Let $\omega^\e$ be the approximating vorticity constructed via the vortex-blob method. Then,
\begin{equation}\label{bound:vb}
\sup_{t\in[0,T]}\int_{\R^2}|\omega^\e(t,x)|(\log(e+|\omega^\e(t,x)|))^\alpha\de x\leq C,
\end{equation}
where $C$ is a positive constant which does not depend on $\e$.
\end{lem}
\begin{proof}
We start by proving that $\omega^\e(0,\cdot)$ satisfies the bound \eqref{bound:vb}, see \eqref{def:init_vb}. We define $\beta(s)=s(\log(e+s))^\alpha$ and let $j_{\delta(\e)}$ be the standard mollifier defined as in \eqref{eq:idv}. Then, defining $\omega_0^\e=\omega_0*j_{\delta(\e)}$, by Jensen's inequality we have that
$$
\int_{\R^2}\beta(|\omega_0*j_{\delta(\e)}|)\de x\leq \int_{\R^2}\beta(|\omega_0|)\de x.
$$
We consider now $\omega_0^\e*\varphi_\e$ where $\varphi_\e$ is the defined as in \eqref{def:vb}. Again by Jensen's inequality we have
$$
\int_{\R^2}\beta(|\omega_0^\e*\varphi_\e|)\de x\leq \int_{\R^2}\beta(|\omega_0|)\de x.
$$
Then, by the properties of the function $\beta$ we have that
$$
\int_{\R^2}\beta(|\omega^\e(0,x)|)\de x \leq C\int_{\R^2}\beta(|\omega^\e(0,x)-\omega_0^\e*\varphi_\e(x)|)\de x+C\int_{\R^2}\beta(|\omega_0^\e*\varphi_\e|)\de x.
$$
We know that the second term on the right hand side is uniformly bounded in $\e$, while for the first term we have that
\begin{align*}
\int_{\R^2}|\omega^\e(0,x)-\omega_0^\e*\varphi_\e(x)|&\log(e+|\omega^\e(0,x)-\omega_0^\e*\varphi_\e(x)|)\de x\\
&\leq \int_{\R^2}|\omega^\e(0,x)-\omega_0^\e*\varphi_\e(x)|\log(e+C\e)\de x \\
&\leq \log(e+C\e)\|\omega^\e(0,\cdot)-\omega_0^\e*\varphi_\e\|_{L^1}\\
&\leq C \e^3\log(e+C\e)\leq C,
\end{align*}
where we have used Lemma \ref{lem:est} with $p=\infty$ in the second line and with $p=1$ in the fourth line. As a consequence of the previous estimate we obtain
$$
\int_{\R^2}|\omega^\e(0,x)|(\log(e+|\omega^\e(0,x)|))^\alpha\de x\leq C.
$$
Let $v^\e$ be the velocity field constructed with the vortex-blob method and consider the linear problem
\begin{equation}
\begin{cases}
\partial_t \bar{\omega}^\e+ v^\e \cdot \nabla \bar{\omega}^\e=0, \\
\bar{\omega}^\e(0,\cdot)=\omega^\e(0,\cdot).
\end{cases}
\label{eq:1}
\end{equation}
By arguing as in the proof of Lemma 3.1 we have that
$$
\int_{\R^2}|\bar{\omega}^\e|(\log(e+|\bar{\omega}^\e|))^\alpha\de x=\int_{\R^2}|\omega^\e(0,x)|(\log(e+|\omega^\e(0,x)|))^\alpha\de x\leq C,
$$
from which it follows that
$$
\int_{\R^2}\beta(|\bar{\omega}^\e*\varphi_\e|(t,x))\de x\leq C.
$$
So, in the end we get that
\begin{align*}
\int_{\R^2}\beta(|\omega^\e|(t,x))\de x&\leq C\int_{\R^2}\beta(|\bar{\omega}^\e*\varphi_\e|(t,x))\de x+C\int_{\R^2}\beta(|\omega^\e-\bar{\omega}^\e*\varphi_\e|(t,x))\de x\\
&\leq C+C \e^3(\log(e+C\e))^\alpha\leq C,
\end{align*}
which concludes the proof.
\end{proof}

\bigskip
\section{Strong convergence of the velocity field}
In this section we will prove that solutions constructed with the three different approximation methods described before are conservative. In particular, the uniform bound proved in Section 3 will be crucial in order to prove the global strong convergence in $C([0,T];L^2(\R^2))$. We start by proving the result for (ES), then with the appropriate modifications we will describe how to prove such result also for (VV) and (VB).
\begin{thm}\label{thm:cons_es}
Let $\omega_0\in L(\log L)^\alpha_c(\R^2)$, with $\alpha>1/2$, which verifies \eqref{zeromean}. Let $u$ be a weak solution of \eqref{eq:eu}, with $\curl u_0=\omega_0$, that can be obtained as a limit of a sequence $u^\delta$ constructed via (ES). Then, $u^\delta$ satisfies the following convergence
\begin{equation}\label{conv:globes}
u^\delta\to u \hspace{1cm}\mbox{in }C([0,T];L^2(\R^2)),
\end{equation}
and $u$ is conservative.
\end{thm}
\begin{proof}
In order to prove the convergence stated in \eqref{conv:globes}, we will prove that $u^\delta$ is a Cauchy sequence in $C([0,T];L^2(\R^2))$. We recall that the parameter $\delta$ is always supposed to vary over a countable set, therefore given the sequence $\delta_n\to 0$, we denote with $u^n$ and $\omega^n$ the sequences $u^{\delta_n}$ and $\omega^{\delta_n}$. We divide the proof in several steps.\\
\\
\underline{Step 1}\hspace{0.5cm}\emph{A Serfati identity with fixed vorticity.}\\
\\
In this step we derive a formula for the approximate velocity $u^n$.\\
Let $a\in C^\infty_c(\R^2)$ be a smooth function such that $a(x)=1$ if $|x|<1$ and $a(x)=0$ for $|x|>2$. Differentiating in time the Biot-Savart formula we obtain that for $i=1,2$
\begin{align}
\partial_s u^n_i(s,x)&=K_i*(\partial_s \omega^n)(s,x) \nonumber\\
&=(aK_i)*(\partial_s \omega^n)(s,x)+[(1-a)K_i]*(\partial_s \omega^n)(s,x).\label{proof:convglob1}
\end{align}
Now we use the equation for $\omega^n$ obtaining
$$
\partial_s \omega^n=-u^n\cdot\nabla \omega^n,
$$
and substituting in \eqref{proof:convglob1} we get
\begin{equation}
\partial_s u^n_i=(aK_i)*(\partial_s \omega^n)-[(1-a)K_i]*(u^n\cdot\nabla \omega^n).
\end{equation}
By using the identity
$$
u^n\cdot\nabla \omega^n=\curl(u^n\cdot\nabla u^n)=\curl\dive(u^n\otimes u^n)
$$
we obtain that
\begin{equation}
[(1-a)K_i]*(u^n\cdot\nabla \omega^n)=\left(\nabla\nabla^\perp[(1-a)K_i]\right)\star(u^n\otimes u^n).\label{proof:convglob2}
\end{equation}
Substituting the expressions \eqref{proof:convglob2} in \eqref{proof:convglob1} and integrating in time we have that $u^n$ satisfies the following formula, known as {\em Serfati identity}:
\begin{equation}\label{eq:serfaties}
\begin{split}
u^n_i(t,x)&=u^n_i(0,x)+(aK_i)*\left(\omega^n(t,\cdot)-\omega^n(0,\cdot)\right)(x)\\
&-\int_0^t\left(\nabla\nabla^\perp[(1-a)K_i]\right)\star(u^n(s,\cdot)\otimes u^n(s,\cdot))(x)\de s.
\end{split}
\end{equation}
We modify the Serfati identity \eqref{eq:serfaties} introducing a new cut-off functions $a_\e$: let $\e\in(0,1)$ and define $a_\e$ to be equal to $1$ on $B_\e$ and $0$ outside  $B_{2\e}$. In this way we rewrite the identity \eqref{eq:serfaties} as
\begin{align*}
u^n_i(t,x)&=u^n_i(0,x)+(a_{\e}K_i)*\left(\omega^n(t,\cdot)-\omega^n(0,\cdot)\right)(x)+(a-a_\e)*K_i]*\left(\omega^n(t,\cdot)-\omega^n(0,\cdot)\right)(x)\\
&-\int_0^t\left(\nabla\nabla^\perp[(1-a)K_i]\right)\star (u^n(s,\cdot)\otimes u^n(s,\cdot))(x)\de s.
\end{align*}
We can prove that $u^n$ is a Cauchy sequence using the previous formula. We consider $u^n,u^m$ with $n,m \in\N$. By linearity of the convolution we have that $u^n-u^m$ satisfies the following
\begin{equation}\label{diff}
\begin{split}
u^n_i(t,x)&-u^m_i(t,x)=\underbrace{u^n_i(0,x)-u^m_i(0,x)}_{(I)}\\ &
+\underbrace{(a_\e K_i)*(\omega^n(t,\cdot)-\omega^m(t,\cdot))(x)}_{(II)}+\underbrace{(a_\e K_i)*(\omega^m_0-\omega^n_0)(x)}_{(III)}\\
& +\underbrace{((a-a_\e)K_i)*(\omega^n(t,\cdot)-\omega^m(t,\cdot))(x)}_{(IV)}+\underbrace{((a-a_\e)K_i)*(\omega^m_0-\omega^n_0)(x)}_{(V)}\\
&-\int_0^t\underbrace{\left(\nabla\nabla^\perp[(1-a)K_i]\right)\star(u^n(s,\cdot)\otimes u^n(s,\cdot)-u^m(s,\cdot)\otimes u^m(s,\cdot))(x)}_{(VI)}\de s.
\end{split}
\end{equation}
In order to prove that $u^n$ is a Cauchy sequence, we fix a parameter $\eta>0$ and we will estimates all the terms in \eqref{diff}. First of all, since the initial datum $u_0^n$ converges strongly in $L^2$, it is obvious that there exists $N_1$ such that $\forall n,m>N_1$
$$
\|u_i^n(0,\cdot)-u_i^m(0,\cdot)\|_{L^2}<\eta.
$$
\underline{Step 2}\hspace{0.5cm}\emph{Estimate on $(VI)$.}\\
\\
By Young's convolution inequality we have that
\begin{align}
\|\nabla\nabla^\perp&[(1-a)K]\star(u^n(s,\cdot)\otimes u^n(s,\cdot)-u^m(s,\cdot)\otimes u^m(s,\cdot)\|_{L^2}\nonumber\\ &\leq \|\nabla\nabla^\perp[(1-a)K]\|_{L^2}\underbrace{\|u^n(s,\cdot)\otimes u^n(s,\cdot)-u^m(s,\cdot)\otimes u^m(s,\cdot)\|_{L^1}}_{(VI*)}\label{arr}.
\end{align}
We add and subtract $u^n(s,\cdot)\otimes u^m(s,\cdot)$ in $(VI*)$ and by H\"older inequality we have
\begin{align*}
\|u^n(s,\cdot)\otimes u^n(s,\cdot)&-u^n(s,\cdot)\otimes u^n(s,\cdot)\|_{L^1}\\ &\leq \left(\|u^n(t,\cdot)\|_{L^2}+\|u^m(t,\cdot)\|_{L^2}\right)\|u^n(s,\cdot)-u^m(s,\cdot)\|_{L^2}.
\end{align*}
For the first factor in \eqref{arr} we have that
$$
\nabla\nabla^\perp[(1-a)K_i]=-(\nabla\nabla^\perp a )K_i-\nabla^\perp a\nabla K_i-\nabla a \nabla^\perp K_i+(1-a)\nabla\nabla^\perp K_i,
$$
and it is easy to see that each term on the right hand side has uniformly bounded $L^2$ norm. Then we have that
\begin{equation}\label{IV}
\begin{split}
\int_0^t\|\nabla\nabla^\perp[(1-a)K]\star(u^n(s,\cdot)&\otimes u^n(s,\cdot)-u^m(s,\cdot)\otimes u^m(s,\cdot)\|_{L^2}\de s\\ & \leq C \|u_0\|_{L^2}\int_0^t\|u^n(s,\cdot)-u^m(s,\cdot)\|_{L^2}\de s.
\end{split}
\end{equation}
\\
\\
\underline{Step 3}\hspace{0.5cm}\emph{Estimate on $(II)$ and $(III)$.}\\
\\
For simplicity we will estimate only $(III)$, but it will be clear from the proof that by using the uniform estimates proved in Section 3 the same estimate holds true for $(II)$. We compute
\begin{align*}
\left\| (a_{\e}K_i)*\right.&\left.(\omega_0^n-\omega_0^m) \right\|_{L^2}^2 =\int_{\R^2}\left| \int_{B_{2\e}(x)}a_\e(x-y)K_i(x-y)\left(\omega_0^n(y)-\omega_0^m(y) \right) \de y \right|^2\de x \\
&\leq \int_{\R^2}\left( \int_{B_{2\e}(x)} \frac{1}{|x-y|}|\omega_0^n(y)-\omega_0^m(y)| \de y \right)^2\de x \\
& = \int_{\R^2}\left( \int_{B_{2\e}(x)} \frac{1}{|x-y|(\log(1/|x-y|))^\alpha}\sqrt{|\omega_0^n(y)-\omega_0^m(y)|(\log(e+|\omega_0^n(y)-\omega_0^m(y)|))^\alpha} \right.\\
& \times \left. \left(\log\left(\frac{1}{|x-y|}\right)\right)^\alpha\sqrt{\frac{|\omega_0^n(y)-\omega_0^m(y)|}{(\log(e+|\omega_0^n(y)-\omega_0^m(y)|))^\alpha}} \de y \right)^2\de x \\
&\leq \int_{\R^2}\int_{B_{2\e}(x)} \frac{1}{|x-y|^2\left(\log(1/|x-y|)\right)^{2\alpha}}|\omega_0^n(y)-\omega_0^m(y)|(\log(e+|\omega_0^n(y)-\omega_0^m(y)|))^\alpha\de y \\
& \times \int_{B_{2\e}(x)} \left(\log\left(\frac{1}{|x-y|}\right)\right)^{2\alpha}\frac{|\omega_0^n(y)-\omega_0^m(y)|}{(\log(e+|\omega_0^n(y)-\omega_0^m(y)|))^\alpha} \de y \de x\\
&\leq \underbrace{\sup_{x} \int_{B_{2\e}(x)} \left(\log\left(\frac{1}{|x-y|}\right)\right)^{2\alpha}\frac{|\omega_0^n(y)-\omega_0^m(y)|}{(\log(e+|\omega_0^n(y)-\omega_0^m(y)|))^\alpha} \de y}_{(I^*)} \\
&\times \underbrace{\int_{\R^2}\int_{B_{2\e}(x)} \frac{1}{|x-y|^2(\log(1/|x-y|))^{2\alpha}}|\omega_0^n(y)-\omega_0^m(y)|(\log(e+|\omega_0^n(y)-\omega_0^m(y)|))^\alpha \de y \de x}_{(II^*)}.
\end{align*}
We estimate $(I^*)$ and $(II^*)$ separately. By defining $\beta(t)=t(\log(e+t))^\alpha$ and
$$
g_\e(x)=\chi_{B_{2\e}}(x)\frac{1}{|x|^2(\log (1/|x|)^{2\alpha}},
$$
we have that
\begin{equation}\label{idea1}
(II^*)=\|g_\e*\beta(|\omega_0^n-\omega_0^m|)\|_{L^1}\leq \|g_\e\|_{L^1}\|\beta(|\omega_0^n-\omega_0^m|)\|_{L^1}.
\end{equation}
By using the convexity of $\beta$ and Lemma 3.1, we have that
$$
\|\beta(|\omega_0^n-\omega_0^m|)\|_{L^1}\leq C,
$$
where $C$ is independent from $n,m$, while for $\alpha>1/2$
\begin{equation}
\|g_\e\|_{L^1}=\frac{C}{(\log(1/\e))^{2\alpha-1}},
\end{equation}
which can be made as small as we want by choosing properly $\e$. For $(I^*)$ we use the following facts on the Legendre transform. Let $\beta(t)=t(\log(e+t))^\alpha$; the maximum of $st-\beta(t)$ occurs at a point $t$ where $s\geq (\log(e+t))^{2\alpha}$, that is, where $t_*(s)\leq e^{s/(2\alpha)}$, so that $\beta^*(s)\leq se^{s/(2\alpha)}$. We apply the inequality
$$
st\leq \Phi(t)+\Phi^*(s)\leq se^{s/(2\alpha)}+t(\log(e+t))^{2\alpha},
$$
to $s=\left(\log\left(\frac{1}{|x-y|}\right)\right)^{2\alpha}$ and $t=\frac{|\omega_0^n-\omega_0^m|}{(\log(e+|\omega_0^n-\omega_0^m|))^\alpha}$ and we find that $(I^*)$ is bounded by 
\begin{align*}
(I^*)\leq &\sup_{x}\left\{ \int_{B_{2\e}} \frac{(\log(1/|z|))^{2\alpha}}{|z|}\de z \right.\\
&\left.+\int_{B_{2\e}(x)}\underbrace{\frac{|\omega_0^n(y)-\omega_0^m(y)|}{(\log(e+|\omega_0^n(y)-\omega_0^m(y)|))^{\alpha}} \log^2\left(e+\frac{|\omega_0^n(y)-\omega_0^m(y)|}{(\log(e+|\omega_0^n(y)-\omega_0^m(y)|))^\alpha}\right)\de y}_{(I^{**})}\right\},
\end{align*}
and we can estimate $(I^{**})$ by
$$
(I^{**})\leq |\omega_0^n-\omega_0^m|(\log(e+|\omega_0^n-\omega_0^m|)^\alpha,
$$
so that $(I^*)$ is finite using the properties of the function $t\mapsto t(\log(e+t))^\alpha$ together with Lemma \ref{lem:es}. So, by fixing $\e$ properly we get that
$$
(II)+(III)\leq C\eta.
$$
\\
\underline{Step 4}\hspace{0.5cm}\emph{Estimates on (IV) and (V).}\\
\\
In the previous step we have fixed the constant $\e$, so by applying Young's inequality to
$$
\|[(a-a_\e)*K_i]*\left(\omega^n_0-\omega_0^m\right)\|_{L^2}\leq \|(a-a_\e)*K_i\|_{L^2}\|\omega^n_0-\omega_0^m\|_{L^1}\leq C(\e)\|\omega^n_0-\omega_0^m\|_{L^1},
$$
$$
\|[(a-a_\e)*K_i]*\left(\omega^n-\omega^m\right)\|_{L^2}\leq \|(a-a_\e)*K_i\|_{L^2}\|\omega^n-\omega^m\|_{L^1}\leq C(\e)\|\omega^n-\omega^m\|_{L^1},
$$
where $C(\e)$ blows up as $\e\to 0$. Now, $\e=\e(\eta)$ has been fixed in the previous step and by Remark \ref{rem:conv_vort} the vorticity converges strongly in $C([0,T],L^1(\R^2))$. Then, we have that there exists $N_2$ such that $\forall n,m>N_2$ 
$$
\|\omega^n_0-\omega_0^m\|_{L^1},\|\omega^n-\omega^m\|_{C(L^1)}<\eta/C(\e).
$$\\
\underline{Step 5}\hspace{0.5cm}\emph{$u^n$ is a Cauchy sequence in $C([0,T];L^2(\R^2))$.}\\
\\
By collecting all the estimates obtained in the previous steps we get that for all $n,m>N:=\max\{N_1,N_2\}$
\begin{equation}
\|u^n(t,\cdot)-u^m(t,\cdot)\|_{L^2}\leq C\left(\eta+\int_0^t\|u^n(s,\cdot)-u^m(s,\cdot)\|_{L^2}\de s\right),
\end{equation}
and by Gronwall's lemma
\begin{equation}\label{proof:cauchyfinal}
\|u^n(t,\cdot)-u^m(t,\cdot)\|_{L^2}\leq C(T)\eta.
\end{equation}
Taking the supremum in time in \eqref{proof:cauchyfinal} we have the result.\\
\\
\underline{Step 6}\hspace{0.5cm}\emph{Conservation of the energy.}\\
\\
Since $u^n$ is an exact smooth solution and smooth solutions are conservative, we have that
\begin{equation}\label{cons_sm}
\|u^n(t,\cdot)\|_{L^2}=\|u_0^n\|_{L^2}.
\end{equation}
Then, since $u^n$ converges strongly to $u$ in $C([0,T];L^2(\R^2))$, by letting $n\to\infty$ in \eqref{cons_sm} we have the result. 
\end{proof}
\vspace{0.5cm}
Now we deal with the vanishing viscosity method.
\begin{thm}\label{thm:cons_vv}
Let $\omega_0\in L(\log L)^\alpha_c(\R^2)$, with $\alpha>1/2$, which verifies \eqref{zeromean}. Let $u$ be a weak solution of \eqref{eq:eu}, with $\curl u_0=\omega_0$, that can be obtained as a limit of a sequence $u^\nu$ constructed via (VV). Then, $u^\nu$ satisfies the following convergence
\begin{equation}\label{conv:globvv}
u^\nu\to u \hspace{1cm}\mbox{in }C([0,T];L^2(\R^2)),
\end{equation}
and $u$ is conservative.
\end{thm}
\begin{proof}
Since the parameter $\nu$ is supposed to vary over a countable set, given the sequence $\nu_n\to 0$, we denote with $u^n$ and $\omega^n$ the sequences $u^{\nu_n}$ and $\omega^{\nu_n}$. Thanks to Remark \ref{rem:mishra}, it is enough to prove that $u^n$ is a Cauchy sequence in $C([0,T];L^2(\R^2))$. We proceeds as in the proof of Theorem \ref{thm:cons_es}. The only difference is that an error term appears in the Serfati identity, which is
\begin{equation}
\int_0^t\left(\Delta[(1-a)K_i]\right)*\left( \nu_n\omega^n(s,\cdot)-\nu_m\omega^m(s,\cdot)\right)\de s.
\end{equation}
By Young's inequality we have that
\begin{align*}
\left\|\left(\Delta[(1-a)K_i]\right)*\left(\nu_n\omega^n(s,\cdot)-\nu_m\omega^m(s,\cdot)\right)\right\|_{L^2}\leq\, &\nu_n\|\Delta[(1-a)K_i]\|_{L^2}\|\omega^n(s,\cdot)-\omega^m(s,\cdot)\|_{L^1}\\
&+\left| \nu_m-\nu_n \right|\|\Delta[(1-a)K_i]\|_{L^2}\|\omega^m(s,\cdot)\|_{L^1},
\end{align*}
Since $\Delta K_i$ is in $L^2(B_1^c)$, a straightforward computation shows that $\Delta[(1-a)K]$ is bounded in $L^2$. So, because of Remark \ref{rem:conv_vort}, there exists $N_3$ such that for all $n,m>N_3$ we have that
$$
\left\|\left(\Delta[(1-a)K]\right)*\left( \nu_n\omega^n(s)-\nu_m\omega^m(s)\right)\right\|_{L^2}\leq C\eta,
$$
and this concludes the proof.
\end{proof}
Finally we deal with the vortex-blob method. The theorem is the following.
\begin{thm}
Let $\omega_0\in L(\log L)^\alpha_c(\R^2)$, with $\alpha>1/2$, which verifies \eqref{zeromean}. Let $u$ be a weak solution of \eqref{eq:eu}, with $\curl u_0=\omega_0$, that can be obtained as the limit of a sequence $u^\e$ constructed via (VB). Then, $u^\e$ satisfies the following convergence
\begin{equation}\label{conv:globvb}
u^\e\to u \hspace{1cm}\mbox{in }C([0,T];L^2(\R^2)),
\end{equation}
and $u$ is conservative.
\end{thm}
\begin{proof}
Since the parameter $\e$ is supposed to vary over a countable set, given the sequence $\e_n\to 0$, we denote with $u^n$ and $\omega^n$ the sequences $u^{\e_n}$ and $\omega^{\e_n}$. We divide the proof in several steps.\\
\\
\underline{Step 1}\hspace{0.5cm}{\em $u^\e$ is a Cauchy sequence in $C([0,T];L^2(\R^2))$.}\\
\\
We proceeds as in the proof of Theorem \ref{thm:cons_es}. The only difference is that an error term appears in the Serfati identity, which is
\begin{equation}
\int_0^t\left(\left(\nabla[(1-a)K_i]\right)\star(F_n(s,\cdot)-F_m(s,\cdot)\right)(x)\de s.
\end{equation}
Since $\nabla[(1-a)K_i]\in L^2(\R^2)$, by using Young's inequality we get that
$$
\|\left(\left(\nabla[(1-a)K_i]\right)\star(F_n(s,\cdot)-F_m(s,\cdot)\right)\|_{L^2}\leq \|\nabla[(1-a)K_i]\|_{L^2}\|F_n(s,\cdot)-F_m(s,\cdot)\|_{L^1},
$$
which can be made as small as we want because of Lemma A.2.\\
\\
\underline{Step 2}\hspace{0.5cm}{\em Conservation of the energy.}\\
\\
We prove now that $u$ is a conservative weak solution. With our notations, multiplying \eqref{eq:vbv} by $u^n$ and integrating in space and time we have that
\begin{equation}\label{eq:energy}
\int_{\R^2}|u^n|^2(t,x)\de x =\int_{\R^2}|u^n|^2(0,x)\de x-\int_0^t\int_{\R^2}(\nabla K\star F_n)\cdot u^n\de x.
\end{equation}
For the second term on the right hand side, by Lemma \ref{lem:fe} we have that
\begin{align*}
\begin{vmatrix}
\displaystyle\int_0^t\int_{\R^2}(\nabla K\star F_n)\cdot u^n\de x
\end{vmatrix}&\leq \|  \nabla K\star F_n(s,\cdot)\|_{L^2}\|  u^n(s,\cdot)\|_{L^2}\\ &\leq \|  F_n(s,\cdot)\|_{L^2}\| u^n(s,\cdot)\|_{L^2}\\ &\leq C(\delta_n)^{-\frac{7}{3}}(\e_n)^{\frac{1}{3}},
\end{align*}
which goes to $0$ as $\e_n\to 0$. Then, by the convergence \eqref{conv:globvb} letting $\e_n\to 0$ in \eqref{eq:energy} we have that 
$$
\int_{\R^2}|u|^2(t,x)\de x =\int_{\R^2}|u_0|^2(x)\de x,
$$
which gives the result.
\end{proof}

\medskip

\subsection*{Acknowledgments}
The author gratefully acknowledge useful discussions with Gianluca Crippa and Stefano Spirito. This work has been started while the author was a PostDoc at the Departement Mathematik und Informatik of the Universit\"at Basel. This research has been partially supported by the ERC Starting Grant 676675 FLIRT.

\bigskip

\appendix
\section{The vortex-blob method}
In this appendix we describe the vortex-blob approximation and some of its properties. Let us consider an initial vorticity $\omega_0\in L^p_c(\R^2)$ with $1\leq p\leq \infty$. Let $\e\in (0,1)$, we consider two small parameters in $(0,1)$, which later will be chosen as functions of $\e$, denoted by $\delta(\e)$ and $h(\e)$.\\
First of all, we consider the lattice
$$
\Lambda_h:=\{\alpha_i\in\Z\times\Z:\alpha_i=h(i_1,i_2), \mbox{ where }i_1,i_2\in\Z  \},
$$
and define $R_i$ the square with sides of lenght $h$ parallel to the coordinate axis and centered at $\alpha_i\in \Lambda_h$. Let $j_{\delta}$ be a standard mollifier and define 
\begin{equation}\label{eq:idv}
\omega_0^\e:=\omega_0*j_{\delta(\e)}.
\end{equation}
For any $\delta\in(0,1)$ the support of $\omega_0^\e$ is contained in a fixed compact set in $\R^2$, then it can be tiled by a finite number $N(\e)$ of squares $R_i$. Define the quantities
$$
\Gamma^\e_i=\int_{R_i} \omega_0^\e(x) \ \de x, \hspace{0.5cm}\mbox{for }i=1,..., N(\e).
$$
Let $\varphi_\e$ be another mollifier, we define the approximate vorticity to be
\begin{equation}\label{def:vb}
\omega^\e(t,x)=\sum_{i=1}^{N(\e)} \Gamma_i^\e \varphi_\e(x-X^\e_i(t)), 
\end{equation}
where $\{X^\e_i(t)\}_{i=1}^{N(\e)}$ is a solution of the O.D.E. system
\begin{equation}\label{eq:vb}
\begin{cases}
\dot{X}^\e_i(t)=u^\e(t,X^\e_i(t)), \\
X^\e_i(0)=\alpha_i,
\end{cases}
\end{equation}
with $u^{\e}$ defined as 
\begin{equation}\label{eq:av}
u^\e(t,x)=K*\omega^\e(t,x)=\sum_{i=1}^{N(\e)} \Gamma_i^\e K_\e(x-X^\e_i(t)),
\end{equation}
where $K_\e=K*\varphi_\e$.
Note that, since $\delta$ and $h$ are $\e$-dependent, we only use the superscript, or subscript, $\e$. The ordinary differential equations \eqref{eq:vb} are known as the {\em vortex-blob approximation}. In particular, the approximation of the initial vorticity and the initial velocity are given by
\begin{equation}\label{def:init_vb}
\omega^\e(0,x)= \sum_{i=1}^{N(\e)} \Gamma_i^\e \varphi_\e(x-\alpha_i),\hspace{0.5cm}u^\e(0,x)=\sum_{i=1}^{N(\e)} \Gamma_i^\e K_\e(x-\alpha_i).
\end{equation}
It is not difficult to show the bound (see \cite{DPM})
\begin{equation}\label{eq:epsvreg}
\sup_{t\in[0,T]}\left( \|u^\e(t,\cdot)\|_{L^\infty}+\|\nabla\,u^\e(t,\cdot)\|_{L^\infty}\right)\leq\frac{C}{\e^2}.
\end{equation}
From \eqref{eq:epsvreg} it follows that, for every fixed $\e>0$, there exists a unique smooth solution $\{X^\e_i(t)\}_{i=1}^{N(\e)}$ of the O.D.E. system \eqref{eq:vb}, which implies that $u^{\e}$ and $\omega^{\e}$ are well-defined smooth functions. Note that $u^\e$ and $\omega^\e$ are not exact solutions of the Euler equations. Precisely, the approximate vorticity $\omega^\e$ satisfies the following equation
\begin{equation}\label{eq:wvb}
\partial_t\omega^\e+u^\e\cdot\nabla\omega^\e=E_\e,
\end{equation}
where by a direct computation the error term is given by
\begin{equation}\label{eq:ee}
E_\e(t,x):=\sum_{i=1}^{N(\e)}\left[u^\e(t,x)-u^\e(t,X_i^\e(t)\right]\cdot \nabla\varphi_\e(x-X^\e_i(t))\Gamma^\e_i.
\end{equation}
Concerning the approximate velocity $u^\e$, consider the quantity
$$
w^\e=\partial_t u^\e+\left(v^\e\cdot\nabla\right) u^\e.
$$
Since $w^\e$ satisfies the system
\begin{equation}
\begin{cases}
\curl w^\e=E_\e,\\
\dive w^\e=\dive \dive \left(u^\e\otimes u^\e \right),
\end{cases}
\end{equation}
we derive that there exists a function $p^\e$ such that
$$
-\Delta p^\e=\dive \dive \left(u^\e\otimes u^\e \right),
$$
and
$$
w^\e=-\nabla p^\e+K*E_\e.
$$
Then, the velocity given by the vortex-blob approximation verifies the following equations
\begin{equation}\label{eq:vbv}
\begin{cases}
\partial_t u^\e+\left(u^\e\cdot\nabla\right) u^\e+\nabla p^\e=K*E_\e,\\
\dive u^\e=0.
\end{cases}
\end{equation}
Since $u^\e$ is divergence-free, $E_\e$ can be rewritten as $E_\e(t,x)=\dive F_\e(t,x)$ where
\begin{equation}\label{eq:fe}
F_\e(t,x):=\sum_{i=1}^{N(\e)}\left[u^\e(t,x)-u^\e(t,X_i^\e(t)\right]\varphi_\e(x-X^\e_i(t))\Gamma^\e_i.
\end{equation}
Let  $\bar{\omega}^\e$ be the solution of the linear transport equation with vector field $u^{\e}$, that is 
\begin{equation}
\begin{cases}
\partial_t \bar{\omega}^\e+ u^\e \cdot \nabla \bar{\omega}^\e=0, \\
\bar{\omega}^\e(0,\cdot)=\omega_0^\e.
\end{cases}
\end{equation}
Since $u^{\e}$ satisfies \eqref{eq:epsvreg}, there exists a unique smooth solution $\bar{\omega}^\e$, which is given by the formula
\begin{equation}\label{eq:baromega}
\bar{\omega}^\e(t,x)= \omega_0^\e((X^\e)^{-1}(t,\cdot)(x)),
\end{equation}
where $X^\e$ is the flow of $u^\e$, that is,
\begin{equation}\label{eq:fv}
\begin{cases}
\dot{X}^\e(t,x)=u^\e(t,X^\e(t,x)), \\
X^\e(0,x)=x.
\end{cases}
\end{equation}
Moreover, since $\dive u^\e=0$, we have
$$
\|\bar{\omega}^\e(t,\cdot)\|_{L^p}=\|\omega_0^\e\|_{L^p}\leq \|\omega_0\|_{L^p}.
$$
The following estimates between the $L^p$ norms of $\omega^\e$ and $\bar{\omega}^\e$ hold true, see \cite{Be,CCS3}.
\begin{lem}\label{lem:est}
Let $\omega_0\in L^{1}(\R^{2})$ and let $h=h(\e)$ be chosen as
\begin{equation}\label{eq:h}
h(\e)=\frac{\e^4}{\exp\left(C_1\e^{-2}\|\omega_0\|_{L^1}T\right)},
\end{equation}
where $C_1>0$ is a positive constant. Then, the estimate
\begin{equation}\label{es:1}
\sup_{0\leq t\leq T} \| \omega^\e-\varphi_\e*\bar{\omega}^\e\|_{L^p} \leq C \e^{1+\frac{2}{p}}
\end{equation}
holds for all $1\leq p\leq \infty$, where $C>0$ is a positive constant which does not depend on $\e$.
\end{lem}
Moreover, with a suitable choice of the parameters in the definition of the vortex-blob method we also have that the error term $F_\e$ goes to $0$ in the limit, see \cite{Be}.
\begin{lem}\label{lem:fe}
Let $\omega_0\in L^p_c(\R^2)$ with $p\geq 1$, then the quantity $F_\e$ defined in \eqref{eq:fe} satisfies
\begin{equation}\label{convFL1}
\sup_{t\in[0,T]}\|F_\e(t,\cdot)\|_{L^1}\to 0, \hspace{0.5cm}\mbox{as }\e\to 0.
\end{equation}
Moreover, choosing $h(\e)=C_1\e^6\exp\left(-C_0\e^{-2}\right)$ where $C_1,C_0$ are positive constants, we have that $F_\e$ satisfies the following additional bound
$$
\|F_\e(t,\cdot)\|_{L^2}\leq C\delta^{-\beta}\e^{\frac{7}{3}}\|\omega_0\|_{L^1},
$$
which goes to $0$ choosing $\delta$ as above and $0<\sigma<1/7$.
\end{lem}
Finally, by showing the equi-integrability of the sequence $\omega^\e$ one of the main results in \cite{CCS3} is the following.
\begin{thm}\label{lem:equi}
Let $\omega_0\in L^1_c(\R^2)$ and $\omega_0^{\e}$ defined as \eqref{eq:idv}. Then the sequence $\omega^\e$ as in \eqref{def:vb} is equi-integrable in $L^{1}((0,T)\times\R^2)$. Moreover, there exists a function $\omega\in C([0,T];L^1(\R^2))$ such that, along a sub-sequence,
$$
\omega^\e\to \omega \hspace{0.5cm}\mbox{ in }C([0,T];L^1(\R^2)),
$$
where $\omega$ is a renormalized and Lagrangian solution of the two-dimensional Euler equations.
\end{thm}

\end{document}